\documentclass{conm-p-l}
\usepackage{amssymb,amsmath,latexsym}

\usepackage{amsfonts}
\usepackage{amscd}

\usepackage{graphicx,array}
\usepackage[all]{xy}
\numberwithin{equation}{section}

\def\ga{\mathfrak{a}}

\def\gh{\mathfrak{h}}

\def\gk{\mathfrak{k}}
\def\gl{\mathfrak{l}}

\def\gs{\mathfrak{s}}
\def\gS{\mathfrak{S}}

\newcommand{\Wn}{W (\fg_n, \fh_{n})}

\renewcommand{\Im}{\mathop{\rm Im} }

\def\C{\mathbb{C}}

\def\H{\mathbb{H}}

\def\N{\mathbb{N}}

\def\R{\mathbb{R}}

\def\Z{\mathbb{Z}}

\newcommand{\id}{\mathrm{id}}

\newcommand{\SO}{\mathrm{SO}}
\newcommand{\Sp}{\mathrm{Sp}}

\newcommand{\SU}{\mathrm{SU}}
\newcommand{\U}{\mathrm{U}}
\newcommand{\diag}{\mathrm{diag}}

\newcommand{\so}{\mathfrak{so}}
\newcommand{\lsp}{\mathfrak{sp}}
\newcommand{\lsl}{\mathfrak{sl}}

\def\cF{\mathcal{F}}

\newcommand{\PW}{\mathrm{PW}}

\newtheorem{theorem}[equation]{Theorem}

\newtheorem{lemma}[equation]{Lemma}

\newtheorem{remark}[equation]{Remark}

\newcommand{\fa}{\mathfrak{a}}

\newcommand{\fg}{\mathfrak{g}}
\newcommand{\fh}{\mathfrak{h}}
\newcommand{\fk}{\mathfrak{k}}
\newcommand{\fl}{\mathfrak{l}}
\newcommand{\fm}{\mathfrak{m}}
\newcommand{\fn}{\mathfrak{n}}

\newcommand{\fp}{\mathfrak{p}}

\newcommand{\fs}{\mathfrak{s}}

\newcommand{\rI}{\mathrm{I}}

\newcommand{\wts}{\widetilde{\sigma}}
\newcommand{\wtW}{\widetilde{W}}
\def\sideremark#1{\ifvmode\leavevmode\fi\vadjust{\vbox to0pt{\vss
 \hbox to 0pt{\hskip\hsize\hskip1em
\vbox{\hsize2cm\tiny\raggedright\pretolerance10000
 \noindent #1\hfill}\hss}\vbox to8pt{\vfil}\vss}}}%

\begin{document}

\title[Restriction of Weyl Group Invariants]{Extension of Symmetric Spaces and Restriction of Weyl Groups and Invariant Polynomials}

\author{Gestur \'{O}lafsson}
\address{Department of Mathematics, Louisiana State University, Baton Rouge,
LA 70803, U.S.A.}
\email{olafsson@math.lsu.edu}
\thanks{The research of G. \'Olafsson was supported by NSF grants
DMS-0402068 and DMS-0801010}

\author{Joseph A. Wolf}
\address{Department of Mathematics, University of California, Berkeley,
CA 94707, U.S.A.}
\email{jawolf@math.berkeley.edu}
\thanks{The research of J. A. Wolf was partially supported by NSF grant
DMS-0652840}


\date{}

\begin{abstract}
Polynomial invariants are fundamental objects in analysis on
Lie groups and symmetric spaces. Invariant differential
operators on symmetric spaces are described by Weyl group
invariant polynomial.
In this article we give a simple criterion that
ensure that the restriction of invariant polynomials
to subspaces is surjective.  In another paper we will apply our criterion to
problems in Fourier analysis on projective/injective limits,
specifically to theorems of Paley--Wiener type.
\end{abstract}
\maketitle

\section*{Introduction} \label{sec0}
\setcounter{equation}{0}

\noindent
Invariant polynomials play a fundamental role in several branches of
mathematics. In this paper we set up the invariant theory needed for
our paper \cite{OW0} on Paley--Wiener theory for injective limits of
Riemannian symmetric spaces.  We also describe that theory, leaving
the proofs of our Paley--Wiener theorems to \cite{OW0}.

Let $G$ be a connected semisimple real Lie group
with Lie algebra $\fg$. Let $\fh\subset \fg$ be a Cartan subalgebra.
Then the algebra of $G$--invariant polynomials on $\fg$ is isomorphic
to the center of the universal enveloping algebra of $\fg$, and
restriction of invariant polynomials to $\fh$ is an isomorphism onto the
algebra of Weyl group invariant polynomials on $\fh$. Replace $G$ by
a Riemannian symmetric space $M=G/K$ corresponding to a
Cartan involution $\theta$ and replace $\fh$  by
a maximal abelian subspace $\fa$ in $\fs:=\{X\in\fg\mid \theta (X)=-X\}$. Then
the Weyl group invariant polynomials correspond to the invariant
differential operators on $M$. They are therefore closely related to
harmonic analysis on $M$, in particular to the determination of
the spherical functions on $M$.

In general we need $\fa \subset \fh$ and $\theta\fh = \fh$.  For this, of
course, we need only choose $\fh$ to be a Cartan subalgebra of the centralizer
of $\fa$.

Denote by $W(\fg,\fh)$ the Weyl group of $\fg$ relative to $\fh$,
$W(\fg,\fa)$ the ``baby'' Weyl group of $\fg$ relative to $\fa$,
$W_\fa (\fg,\fh)=\{w\in W(\fg,\fh)\mid w(\fa )= \fa\}$, $\rI(\fg )$
the algebra of $W(\fg,\fh)$--invariant polynomials on $\fh$ and
finally $\rI(\fa )$ the algebra of $W(\fg,\fa)$--invariant polynomials on
$\fa$.   It is well known for all semisimple
Lie algebras that $W_\fa (\fg,\fh)|_{\fa}=W (\fg,\fa)$.  In \cite{He1964}
Helgason showed for all classical semisimple Lie algebras
that $\rI(\fh)|_\fa = \rI(\fa)$. As an application, this shows that in most
cases the invariant differential operators on $M$ come from elements
in the center of the universal enveloping algebra of $\fg$.

In this article we discuss similar restriction problems
for the case of pairs of Lie groups $G_n\subset G_k$ and
symmetric spaces $M_n\subset M_k$. We use the above notation with
indices ${}_n$ respectively ${}_k$. The first question is
about restriction from $\fh_k$ to $\fh_n$. It is clear that neither
does the group $W_{\fh_n}(\fg_k,\fh_k)$ restrict to $W (\fg_n,\fh_n)$ in
general, nor is $\rI(\fh_k)|_{\fh_n}=\rI (\fh_n)$. To make this work, we
introduce the notion that $\fg_k$ is a propagation of $\fg_n$ using
the Dynkin diagram of simple Lie classical Lie algebras.

In terms of restricted
roots, propagation means that either the rank and restricted root system
of the large and the small symmetric spaces are the same, or
roots are added to the left end of the Dynkin diagram.  The result is
that both symmetric spaces have the same type of root system but the
larger one can have higher rank. In that case the restriction
result above holds for all cases except when the restricted root systems
are of type $D$. This includes
all the cases of classical Lie groups of the same type. If $G_k$ is
a propagation of $G_n$, then $W_{\fh_n}(\fg_k,\fh_k)|_{\fh_n}=W (\fg_n,\fh_n)$ and
$\rI(\fh_n)|_{\fh_n}=\rI (\fh_n)$, except in the case of simple algebras
of type $D$, where a parity condition is needed, i.e., we have to extend
the Weyl group by incorporating odd sign changes for simple factors of type $D$.
The resulting finite group is denoted by $\widetilde{W} (\fg,\fh)$. Then,
in all classical cases, the
$\widetilde{W}(\fg_k,\fh_k)$-invariant polynomials restrict to
$ \widetilde{W} (\fg_n,\fh_n)$-invariant polynomials. We also
show that $\widetilde{W}_\fa(\fg,\fh)|_\fa = \widetilde{W}(\fg,\fa)$.

In Section \ref{sec2} we introduce the notion of
propagation and examine the corresponding invariants explicitly for each
type of root system.  The main result, Theorem \ref{th-AdmExt},
summarizes the facts on restriction of Weyl groups for
propagation of symmetric spaces. The proof is by case by case consideration
of each simple root system.

In Section \ref{sec3} we prove surjectivity of Weyl group invariant
polynomials for propagation of symmetric spaces.  As mentioned above,
this is analogous to Helgason's result on restriction of invariants from
the full Cartan $\fh$ of $\fg$ to the Cartan $\fa$ of $(\fg,\fk)$.

In Section \ref{sec1} we indicate some applications of our results on 
Weyl group invariants to
Fourier analysis on Riemannian symmetric spaces of noncompact type. 
This includes applications to the Fourier transform of compactly supported 
functions and the Paley-Wiener theorem as well as applications to invariant 
differential operators and related differential equations on symmetric 
spaces and their inductive limits.

\section{Restriction of Invariants for Classical Simple Lie Algebras}
\label{sec2}
\noindent
In this section we discuss restriction of polynomial functions invariant under a Weyl group of
classical type, i.e., a finite reflection group associated to a classical root system. Those can be concretely realized as permutation groups extended by a group of sign changes.

Let $\fg_n$ be a simple Lie algebra of classical type
and let $\fh_n\subset \fg_n$ be a Cartan subalgebra. Let
$\Delta_n = \Delta(\fg_n, \fh_n)$ be the set of roots of $\fh_{n,\C}$ in
$\fg_{n,\C}$ and $\Psi_n = \Psi(\fg_n, \fh_n)$ a set of simple roots. We label
the corresponding
Dynkin diagram so that $\alpha_1$ is the \textit{right} endpoint.
If $\fg_n\subseteqq \fg_k$ then we chose $\fh_n$ and $\fh_k$ so
that $\fh_n=\fg_n\cap \fh_k$.
We say that $\fg_k$ \textit{propagates} $\fg_n$, if
$\Psi_k$ is constructed from $\Psi_n$ by adding simple roots to the \textit{left} end
of the Dynkin diagrams:
\begin{equation}\label{rootorder}
\begin{aligned}
&\begin{tabular}{|c|l|c|}\hline
$\Psi_k=A_k$&
\setlength{\unitlength}{.5 mm}
\begin{picture}(155,18)
\put(5,2){\circle{2}}
\put(2,5){$\alpha_{k}$}
\put(6,2){\line(1,0){13}}
\put(24,2){\circle*{1}}
\put(27,2){\circle*{1}}
\put(30,2){\circle*{1}}
\put(34,2){\line(1,0){13}}
\put(48,2){\circle{2}}
\put(45,5){$\alpha_n$}
\put(49,2){\line(1,0){23}}
\put(73,2){\circle{2}}
\put(70,5){$\alpha_{n-1}$}
\put(74,2){\line(1,0){23}}
\put(98,2){\circle{2}}
\put(95,5){$\alpha_{n-2}$}
\put(99,2){\line(1,0){13}}
\put(117,2){\circle*{1}}
\put(120,2){\circle*{1}}
\put(123,2){\circle*{1}}
\put(129,2){\line(1,0){13}}
\put(143,2){\circle{2}}
\put(140,5){$\alpha_1$}
\end{picture}
&$k\geqq n \geqq 1$
\\
\hline
\end{tabular}\\
&\begin{tabular}{|c|l|c|}\hline
$\Psi_k=B_k$&
\setlength{\unitlength}{.5 mm}
\begin{picture}(155,18)
\put(5,2){\circle{2}}
\put(2,5){$\alpha_{k}$}
\put(6,2){\line(1,0){13}}
\put(24,2){\circle*{1}}
\put(27,2){\circle*{1}}
\put(30,2){\circle*{1}}
\put(34,2){\line(1,0){13}}
\put(48,2){\circle{2}}
\put(45,5){$\alpha_n$}
\put(49,2){\line(1,0){23}}
\put(73,2){\circle{2}}
\put(70,5){$\alpha_{n-1}$}
\put(74,2){\line(1,0){13}}
\put(93,2){\circle*{1}}
\put(96,2){\circle*{1}}
\put(99,2){\circle*{1}}
\put(104,2){\line(1,0){13}}
\put(118,2){\circle{2}}
\put(115,5){$\alpha_2$}
\put(119,2.5){\line(1,0){23}}
\put(119,1.5){\line(1,0){23}}
\put(143,2){\circle*{2}}
\put(140,5){$\alpha_1$}
\end{picture}
&$k\geqq n \geqq 2$\\
\hline
\end{tabular} \\
&\begin{tabular}{|c|l|c|}\hline
$\Psi_k=C_k$ &
\setlength{\unitlength}{.5 mm}
\begin{picture}(155,18)
\put(5,2){\circle*{2}}
\put(2,5){$\alpha_{k}$}
\put(6,2){\line(1,0){13}}
\put(24,2){\circle*{1}}
\put(27,2){\circle*{1}}
\put(30,2){\circle*{1}}
\put(34,2){\line(1,0){13}}
\put(48,2){\circle*{2}}
\put(45,5){$\alpha_n$}
\put(49,2){\line(1,0){23}}
\put(73,2){\circle*{2}}
\put(70,5){$\alpha_{n-1}$}
\put(74,2){\line(1,0){13}}
\put(93,2){\circle*{1}}
\put(96,2){\circle*{1}}
\put(99,2){\circle*{1}}
\put(104,2){\line(1,0){13}}
\put(118,2){\circle*{2}}
\put(115,5){$\alpha_2$}
\put(119,2.5){\line(1,0){23}}
\put(119,1.5){\line(1,0){23}}
\put(143,2){\circle{2}}
\put(140,5){$\alpha_1$}
\end{picture}
& $k\geqq n \geqq 3$
\\
\hline
\end{tabular}\\
&\begin{tabular}{|c|l|c|}\hline
$\Psi_k=D_k$ &
\setlength{\unitlength}{.5 mm}
\begin{picture}(155,20)
\put(5,9){\circle{2}}
\put(2,12){$\alpha_{k}$}
\put(6,9){\line(1,0){13}}
\put(24,9){\circle*{1}}
\put(27,9){\circle*{1}}
\put(30,9){\circle*{1}}
\put(34,9){\line(1,0){13}}
\put(48,9){\circle{2}}
\put(45,12){$\alpha_n$}
\put(49,9){\line(1,0){23}}
\put(73,9){\circle{2}}
\put(70,12){$\alpha_{n-1}$}
\put(74,9){\line(1,0){13}}
\put(93,9){\circle*{1}}
\put(96,9){\circle*{1}}
\put(99,9){\circle*{1}}
\put(104,9){\line(1,0){13}}
\put(118,9){\circle{2}}
\put(113,12){$\alpha_3$}
\put(119,8.5){\line(2,-1){13}}
\put(133,2){\circle{2}}
\put(136,0){$\alpha_1$}
\put(119,9.5){\line(2,1){13}}
\put(133,16){\circle{2}}
\put(136,14){$\alpha_2$}
\end{picture}
& $k\geqq n \geqq 4$
\\
\hline
\end{tabular}
\end{aligned}
\end{equation}

Let $\fg$ and ${}'\fg\subset \fg$ be semisimple Lie algebras. Then
$\fg$ \textit{propagates} ${}'\fg$
if we can number the simple ideals $\fg_j$, $j=1,2, \ldots ,r$, in $\fg$ and
the simple ideals ${}'\fg_i$, $i = 1, 2, \dots , s$, in ${}'\fg$, so that
$\fg_j$ propagates ${}'\fg_j$ for $j=1,\ldots ,s$.

When $\fg_k$ propagates $\fg_n$ as above, they have Cartan subalgebras
$\fh_k$ and $\fh_n$ such that $\fh_n\subseteqq \fh_k$, and we have
choices of root order such that
\[\text{if } \alpha \in \Psi_n \text{ then there is a unique }
\alpha' \in \Psi_k \text{ such that } \alpha'|_{\fh_n} = \alpha.\]
It follows that
\[\Delta_n \subseteqq \{\alpha|_{\fh_n}\mid \alpha\in \Delta_k \text{ and }
\alpha|_{\fh_n}\not= 0\}\, .\]

For a Cartan subalgebra $\fh_\C$ in a semisimple complex Lie algebra $\fg_\C$
denote by $\fh_\R$ the Euclidean vector space
\[
\fh_\R=\{X\in\fh_\C\mid
\alpha (X) \in \R \text{ for all } \alpha \in \Delta (\fg_\C,\fh_\C)\}\, .
\]

We now discuss case by case the classical simple Lie algebras
and how the Weyl group and the invariants behave under
propagation. The result will be collected in
Theorem \ref{th-AdmExt} below. The corresponding result for
Riemannian symmetric spaces is Theorem \ref{th-AdmExtG/K}.

For $s\in \N$ identify $\R^s$ with its dual. Let $f_1= (0, 0, \ldots ,0,1) $,
\ldots , $f_s=(1,0,0,\ldots ,0)$ be the standard
basis for $\R^s$.  This enumeration is opposite to the usual one.  We write
\[x=x_1f_1+\ldots +x_s f_s=(x_s,\ldots ,x_1)\]
to indicate that in the following we will be adding zeros to the left
to adjust for our numbering in the Dynkin diagrams. We use the discussion in
\cite[p. 293]{V1974} as a reference for the realization of
the classical Lie algebras.

When $\fg$ is a classical simple Lie algebra of rank $n$ we write $\pi_n$ for the
defining representation and
\[F_n(t,X):=\det (t+\pi_n (X))\, .\]
We denote by the same letter the restriction of $F_n(t, \cdot )$ to $\fh_n$.
In this section only we use the following simplified notation:
$W_k = W(\fg_k, \fh_k)$ denotes the usual Weyl group of the pair $(\fg_k,\fh_k)$ and
$W_{k,n}=
W_{\fh_{n,\R}}(\fg_k, \fh_{k})=\{w\in W_k\mid w(\fh_{n,\R})=\fh_{n,\R}\}$
is the subgroup with well defined restriction to $\fh_n$.

\noindent{\textbf{The case
$\mathbf{A_k}$, where $\mathbf{\fg_k =\mathfrak{sl}(k+1,\C)}$.}} In this case
\begin{equation}\label{an}
\fh_{k,\R}=\{(x_{k+1},\ldots ,x_{1})\in \R^{k+1}\mid x_1+\ldots +x_{k+1}=0\}\, ,
\end{equation}
where $x\in \R^{k+1}$ corresponds to the diagonal matrix
\[x\leftrightarrow \mathrm{diag}(x):=\left ( \begin{smallmatrix} x_{k+1} &0 &
\ldots &0 \cr 0 &x_{k} & & \cr & &  \ddots & \cr &&& x_1\end{smallmatrix}\right )\]
Then $\Delta =\{ f_i-f_j\mid 1\leqq i\not= j\leqq k+1\}$ where $f_\ell$
maps a diagonal matrix to its $\ell^{th}$ diagonal element.  Here
$W (\fg_k,\fh_k)$ is the symmetric group $\gS_{k+1}$, all permutations of
$\{1, \dots , k+1\}$, acting on the $\fh_{k}$ by
\[\sigma\cdot (x_{k+1},\ldots ,x_{1})=
(x_{\sigma^{-1}(k+1)},\ldots ,x_{\sigma^{-1}(1)})\, .\]
We will use the simple root system
\[\Psi (\fg_k,\fh_k) =\{f_j-f_{j-1}\mid j=2,\ldots ,k+1\}\, .\]
The analogous notation will be used for $A_n$. In particular, denoting
the zero vector of length $j$ by $0_j$, we have
\begin{equation}\label{eq-hnInhk}
\fh_{n,\R}=\left \{ (0_{k-n},x_{n+1},\ldots ,x_1)\mid  x_j\in\R\quad
\text{and}\quad \sum_{j=1}^{n+1}x_j=0 \right \}\subset \fh_{k,\R}\, .
\end{equation}
This corresponds to the embedding
\[\mathfrak{sl} (n,\C)\hookrightarrow \mathfrak{sl} (k,\C)\, ,\quad X\mapsto
\left ( \begin{smallmatrix} 0_{k-n,k-n} & 0\cr
0 & X\end{smallmatrix}\right )\, .\]
It follows that
\[W_{k,n}=\gS_{k-n}\times \gS_{n+1}\, .\]
Hence $W_{k,n}|_{\fh_{n,\R}}=W (\fg_n,\fh_n)$
and the kernel of the restriction map is the
first factor $\gS_{k-n}$.

According to \cite[Exercise 58, p. 410]{V1974} we have
\[F_k(t,X)=\prod_{j=1}^{k+1}(t+x_j)=
t^{k+1}+\sum_{ \nu=1}^{k+1} p_{k, \nu} (X)t^{\nu -1}\, .\]
The polynomials $p_{k,\nu}$ generate $\rI_{W (\fg_k,\fh_k)}(\fh_{k,\R})$.
By (\ref{eq-hnInhk}), if $X=(0_{k-n},x)\in \fh_{n,\R}$, then
$$
\begin{aligned}
F_k(t,(0_{k-n},x))&= t^{k+1}+\sum_{\nu=1}^{k+1} p_{k, \nu} (X)t^{\nu -1}
     = t^{k-n}\det (t+\pi_{n}(x))\\
&= t^{k-n}(t^{n+1} +\sum_{\nu =1}^{n+1} p_{n,\nu }(x)t^{\nu -1})
     = t^{k+1}+\sum_{\nu =k-n+1}^{k+1} p_{n,\nu +n-k}(x)t^{\nu -1}\, .
\end{aligned}
$$
Hence
\[p_{k,\nu}|_{\fh_{n,\R}}= p_{n,\nu +n-k} \text{ for }
k-n+1 \leqq \nu \leqq k\]
and
\[p_{k,\nu}|_{\fh_{n,\R}}=0 \text{ for } 1 \leqq \nu \leqq k-n\, .\]
In particular the restriction map $\rI_{W (\fg_k,\fh_k)}(\fh_{k,\R})\to \rI_{W (\fg_n,\fh_n)}(\fh_{n,\R})$
is surjective.

\noindent
\textbf{The case $\mathbf{B_k}$, where $\mathbf{\fg_k=\so (2k+1,\C)}$.} In this case
$\fh_{k,\R}=\R^k$ where $\R^k$ is embedded into $\so (2k+1,\C)$ by
\begin{equation}\label{bn}
x\mapsto \left ( \begin{smallmatrix} 0 & 0 &
   0\cr 0 &\mathrm{diag}(x)& 0\cr 0 & 0 &-\mathrm{diag}(x)\end{smallmatrix}\right )\, .
\end{equation}
Here $\Delta_k=\{\pm (f_i\pm f_j)  \mid 1\leqq j < i \leqq k\}\bigcup
\{\pm f_1,\ldots ,\pm  f_k\}$ and we have the positive system
$\Delta_k^+=\{f_i\pm f_j\mid 1\leqq j< i\leqq k\}\bigcup \{f_1,\ldots ,f_k\}$.
The simple root system is
$\Psi = \Psi(\fg_k,\fh_k) = \{\alpha_1, \dots , \alpha_k\}$ where
\[\text{the simple root } \alpha_1=f_1 \text{, and } \alpha_j=f_{j}-f_{j-1}
\text{ for } 2 \leqq j \leqq k.\]
In this case the Weyl group $W (\fg_k,\fh_k)$ is the semidirect product
$\gS_k\rtimes \{1,-1\}^k$, where
$\gS_k$ acts as before and
\[\{1,-1\}^k\cong (\Z /2\Z )^k=\{\mathbf{\epsilon}=(\epsilon_k,\ldots ,\epsilon_1)\mid \epsilon_j=\pm 1\}\]
acts by sign changes,
$\mathbf{\epsilon}\cdot x=(\epsilon_k x_k,\ldots ,\epsilon_1 x_1)\, .$
Similar notation holds for $\fh_{n,\R}$. Our embedding of
$\fh_{n,\R}\hookrightarrow \fh_{k,\R}$ corresponds to the (non-standard) embedding of
$\so (2n+1,\C)$ into $\so (2k+1,\C)$ given by
\[\begin{pmatrix} 0 & a & b \cr -b^t & A & B\cr -a^t & C & -A^t\end{pmatrix}
\mapsto \left ( \begin{smallmatrix}0 & 0_{k-n} & a & 0_{k-n} & b  \cr 0_{k-n}^t & 0  & 0 & 0 &
0 \cr
-b^t & 0  & A & 0 & B\cr
0_{k-n}^t & 0  & 0 & 0 &
0\cr
-a^t & 0 & C & 0 & -A^t
\end{smallmatrix} \right )
\]
where the zeros stands for the zero matrix of the obvious size and we use
the realization from \cite[p. 303]{V1974}.  Here we see that
\[W_{k,n}=(\gS_{k-n}\rtimes \{1,-1\}^{k-n})\times
(\gS_n\rtimes \{1,-1\}^n)\, .\]
Thus $W_{k,n}|_{\fh_{n,\R}}= W (\fg_n,\fh_n)$ and the kernel
of the restriction map is
$\gS_{k-n}\rtimes \{1,-1\}^{k-n}$.

For the invariant polynomials we have, again using
\cite[Exercise 58, p. 410]{V1974}, that
\[F_k(t,X)= \det(t + \pi_k(X)) =
t^{2k+1}+\sum_{\nu =1}^k p_{k,\nu} (X)t^{2\nu -1}\]
and the polynomials $ p_{k,\nu}$ freely generate $\rI_{W(\fg_k,\fh_k)}(\fh_{k,\R})$.
For $X \in \fh_k$,  $F_k(t,X)$ is given by
$t\prod_{j=1}^n(t+x_j)(t-x_j)=t\prod_{j=1}^n(t^2-x_j^2)$.
Arguing as above we have
for $X=(0_{k-n},x)\in \fh_{n,\R}\subseteqq \fh_{k,\R}$:
$$
\begin{aligned}
F_k(t,(0_{k-n},x))&=t^{2k+1}+\sum_{\nu=1}^k p_{k, \nu} (X)t^{2\nu -1}
  = t^{2(k-n)}\det (t+\pi_{n}(x))\\
&=t^{2(k-n)}(t^{2n+1} +\sum_{\nu =1}^n p_{n,\nu }(x)t^{2\nu -1})
  = t^{2k+1}+\sum_{\nu =k-n+1}^k p_{n,\nu +n-k}(x)t^{2\nu -1}\, .
\end{aligned}
$$
Hence
\[p_{k,\nu}|_{\fh_{n,\R}}= p_{n,\nu +n-k} \text{ for }
k-n+1 \leqq \nu \leqq k \]
and
\[p_{k,\nu}|_{\fh_{n,\R}}=0 \text{ for } 1 \leqq \nu \leqq k-n\, .\]
In particular, the restriction map $\rI_{W (\fg_k,\fh_k)}(\fh_{k,\R})\to \rI_{W (\fg_n,\fh_n)}(\fh_{n,\R})$
is surjective.

\noindent
\textbf{The case $\mathbf{C_k}$, where $\mathbf{\fg_k=\lsp (k,\C)}$.} Again
$\fh_{k,\R}=\R^k$ embedded in $\lsp (k,\C)$ by
\begin{equation}\label{cn}
x\mapsto \begin{pmatrix} \diag (x) & 0 \cr 0 &-\diag (x)\end{pmatrix}\, .
\end{equation}
In this case
$\Delta_k=\{\pm (f_i\pm f_j)  \mid 1\leqq j < i \leqq k\}\bigcup \{\pm 2f_1,\ldots ,\pm 2f_k\}\, .$
Take $\Delta_k^+=\{f_i-f_j\mid 1\leqq j< i\leqq k\}\cup \{2f_1,\ldots ,2 f_k\}$
as a positive system. Then the simple root system
$\Psi = \Psi (\fg_k,\fh_k) = \{\alpha_1 , \dots , \alpha_k\}$ is given by
\[\text{the simple root } \alpha_1=2f_1 \text{, and } \alpha_j=f_{j}-f_{j-1}
\text{ for } 2 \leqq j \leqq k.\]
The Weyl group $W(\fg_k,\fh_k)$ is again $\gS_k\rtimes \{1,-1\}^k$ and
\[W_{k,n}=(\gS_{k-n}\rtimes \{1,-1\}^{k-n})\times (\gS_n\rtimes\{1,-1\}^n)\, .\]
Thus, $W_{k,n}|_{\fh_{n,\R}}= W (\fg_n,\fh_n)$ and the kernel of the restriction map is
$\gS_{k-n}\rtimes \{1,-1\}^{k-n}$.

For the invariant polynomials we have, again using
\cite[Exercise 58, p. 410]{V1974}, that
\[F_k(t,X)=
t^{2k}+\sum_{\nu =1}^k p_{k,\nu} (X)t^{2(\nu -1)} =
\prod_{j=1}^n (t^2-x_j^2)\]
and the $ p_{k,\nu}$ freely generate $\rI_{W (\fg_k,\fh_k)}(\fh_{k,\R})$.
We embed $\lsp (n,\C)$ into $\lsp (k,\C)$ by
\[\begin{pmatrix} A & B \cr C & -A^t\end{pmatrix}
\mapsto \left (\begin{smallmatrix} 0_{k-n,k-n} & 0 & 0  & 0 \cr
0& A &0 & B\cr
0 & 0 & 0_{k-n,k-n} & 0\cr
0 & C& 0 &-A^t\end{smallmatrix}\right )\]
where as usual $0$ stands for a zero matrix of the correct size.
Then
$$
\begin{aligned}
F_k(t,(0_{k-n},x))&= t^{2k}+\sum_{\nu=1}^k p_{k, \nu} (X)t^{2(\nu -1)}
  = t^{2(k-n)}\det (t+\pi_{n}(x))\\
&= t^{2(k-n)}(t^{2n} +\sum_{\nu =1}^n p_{n,\nu }(x)t^{2(\nu -1)})
  = t^{2k}+\sum_{\nu =k-n+1}^k p_{n,\nu +n-k}(x)t^{2(\nu -1)}\, .
\end{aligned}
$$
Hence
\[p_{k,\nu}|_{\fh_{n,\R}}= p_{n,\nu +n-k} \text{ for }
k-n+1 \leqq \nu \leqq k\]
and
\[p_{k,\nu}|_{\fh_{n,\R}}=0 \text{ for } 1 \leqq \nu \leqq k-n\, .\]
In particular, the restriction map
$\rI_{W (\fg_k,\fh_k)}(\fh_{k,\R})\to \rI_{W (\fg_n,\fh_n)}(\fh_{n,\R})$ is surjective.

\noindent
\textbf{The case $\mathbf{D_k}$, where $\mathbf{\fg_k=\so (2k,\C)}$.} We take
$\fh_{k,\R}=\R^k$ embedded in $\so (2k,\C)$ by
\begin{equation}\label{dn}
x\mapsto \begin{pmatrix} \diag(x) & 0 \cr 0 & -\diag (x)\end{pmatrix}\, .
\end{equation}
Then
$\Delta_k = \{\pm (f_i\pm f_j)\mid 1\leqq j<i\leqq k\}$ and we use the simple
root system $\Psi (\fg_k,\fh_k)= \{\alpha_1 , \dots , \alpha_k\}$ given by
\[ \alpha_1 = f_1+f_2, \text{ and } \alpha_i = f_i-f_{i-1} \text{ for }
2 \leqq i \leqq k\]
The Weyl group is
$W (\fg_k,\fh_k)=\gS_k\rtimes
  \{\mathbf{\epsilon}\in \{1,-1\}^k\mid \epsilon_1 \cdots \epsilon_=1\}\, .$
In other words the elements of $W(\fg_k,\fh_k)$ contain only an
\textit{even} number of sign-changes.  The invariants are given by
\[F_k (t,X)
= t^{2k}+\sum_{\nu =2}^{k} p_{k,\nu }(X) t^{2(\nu-1)}+p_{k,1}(X)^2
= \prod_{\nu =1}^n (t^2-x_j^2)\]
where $p_1$ is the Pfaffian, $p_1(X)=(-1)^{k/2}x_1\ldots x_k$, so
$p_1(X)^2=\det (X)$.  The polynomials $p_{k,1},\ldots ,p_{k,k}$
freely generate $\rI_{W (\fg_k,\fh_k)}(\fh_{k,\R})$.

We embed $\fh_{n,\R}$ in $\fh_{k,\R}$ in the same manner as before.
This corresponds to
\[\begin{pmatrix} A & B\cr C & -A^t\end{pmatrix}\mapsto
\left ( \begin{smallmatrix} 0_{k-n,k-n} & & 0_{k-n,k-n} & \cr
0 &A & 0  & B\cr
0_{k-n,k-n} & 0 & 0 \cr
0 & C & 0 & -A^t\end{smallmatrix}\right )
\, .
\]
It is then clear that
\[W_{k,n}=(\gS_{k-n}\rtimes \{1,-1\}^{k-n})\times_*
(\gS_n\rtimes \{1,-1\}^{n})\]
where the ${}_*$ indicates that $\epsilon_1\cdots \epsilon_n=1$.
Therefore, the restrictions of elements of $W_{k,n}$, $k>n$, contain
all sign changes, and
\[\gS_n\rtimes \{\epsilon \in \{1,-1\}^{n-1}\mid \epsilon_1\ldots \epsilon_n=1\} =
W (\fg_n,\fh_n) \subsetneqq W_{k,n}|_{\fh_{n,\R}}=\gS_n\rtimes \{1,-1\}^n\, .\]
The Pfaffian $p_{k,1}(0,X)=0$ and
\begin{eqnarray*}
F_k(t,(0,x))&=& t^{2k}+\sum_{\nu =2}^{k} p_{k,\nu }(0,x) t^{2(\nu-1)}\\
&=& t^{2(k-n)}F_n(t,x) = t^{2(k-n)}(t^{2n}+
\sum_{\nu =2}^{n} p_{n,\nu }(x) t^{2(\nu-1)} + p_{n,1}(x)^2)\\
&=& t^{2k}+\sum_{\nu=k-n+2}^k p_{n,\nu +n-k}(x)t^{2(\nu -1)}+
p_{n,1}(x)^2t^{2(k-n)}\, .
\end{eqnarray*}
Hence
$$
\begin{aligned}
&p_{k,\nu}|_{\fh_{n,\R}} = p_{n,\nu +n-k} \text{ for }
   k-n+2\leqq \nu \leqq k\, , \\
&p_{k, k-n+1}|_{\fh_{n,\R}}=p_{n,1}(x)^2\,, \text{ and } \\
&p_{k,\nu}|_{\fh_{n,\R}}=0\, ,\quad \nu =1,\ldots , k-n\, .
\end{aligned}
$$
In particular the elements in $\rI_{W(\fg_k,\fh_k)}(\fh_{k,\R})|_{\fh_{n,\R}}$ are
polynomials in even powers of $x_j$ and $p_{n,1}$ is not in the image of
the restriction map. Thus
\[\rI_{W (\fg_k,\fh_k)}(\fh_{k,\R})|_{\fh_{n,\R}}\subsetneqq
 \rI_{W (\fg_n,\fh_n)}(\fh_{n,\R})\, .\]

Let $\sigma_k$ be the involution of the Dynkin diagram for $D_k$ given by $\sigma (\alpha_1)=\alpha_2$, $\sigma (\alpha_2)=\alpha_1$ and $\sigma_k(\alpha_j)=\alpha_j$ for $3\le j\le k$. Then $\sigma_k|_{\fh_n}=\sigma_n$, $\sigma_k(\fh_{n,\R})$ and $\sigma_k$ normalizes $W (\fg_k,\fh_k)$. The group $\widetilde{W}_k=\widetilde{W} (\fg_k,\fh_k):=W (\fg_k,\fh_k)\rtimes \{1,\sigma_k\}$ is the group $\gS_k\rtimes \{1,-1\}^k$. Hence
\[\widetilde{W} (\fg_n,\fh_n)=W_{\fh_n} (\fg_k,\fh_k)|_{\fh_{n,\R}}=
\widetilde{W}_{\fh_n}(\fg_k,\fh_k)|_{\fh_n}\, .\]
We also note that $\widetilde{W}(\fg_k,\fh_k)$ is isomorphic to the  
Weyl group of the root system $B_k$ and hence is a finite reflection group.

The algebra $\rI_{\widetilde{W}_k}(\fg_k,\fh_k)$ is the algebra of all \textit{even} elements in $\rI_{W_k}(\fg_k,\fh_k)$. Denote it by $\rI_{W(\fg_k,\fh_k)}^{\text{even}}(\fh_{k,\R})$. The above calculations shows that
\[\rI_{W_n} ^{\text{even}}(\fg_n,\fh_n)=\rI_{W_k}(\fg_k,\fh_k)|_{\fh_n}=\rI_{W_k} ^{\text{even}}(\fg_k,\fh_k)|_{\fh_n}\, .\]

We put these results together in the following theorem.

\begin{theorem}\label{th-AdmExt}  Assume $\fg_n$ and $\fg_k$ are
simple complex Lie algebras of ranks $n$ and $k$, respectively, and
that $\fg_k$ propagates $\fg_n$.
\begin{enumerate}
\item
If $\fg_n\not=\so (2n,\C)$ and $\fg_k\not= \so (2k,\C)$ then
\[
W(\fg_n,\gh_n) = W_{\gh_n}(\fg_k,\fh_k)|_{\fh_n}
= \{w|_{\fh_n} \mid w \in W(\fg_k,\fh_k) \text{ with } w(\fh_n) = \fh_n\}
\]
and the restriction map
\[\rI_{W(\fg_k,\fh_k)}(\fh_{k,\R})\to \rI_{W(\fg_n,\fh_n)}(\fh_{n,\R})\]
is surjective.

\item If $\fg_n=\so (2n,\C)\subset \fg_k=\so(2k,\C)$, then
\[
\text{\phantom{XX}} W_{\gh_n}(\fg_k,\fh_k)|_{\fh_n} =\{w|_{\fh_n}\mid w
\in W(\fg_k,\fh_k) \text{ with } w(\fh_n)=\fh_n\}=\gS_n\rtimes \{1,-1\}^n
\]
contains all sign changes, but the elements of $\Wn$ contain only even
numbers of sign changes, so
$W(\fg_n,\gh_n) \subsetneqq W_{\gh_n}(\fg_k,\fh_k)|_{\fh_n}$.  The
elements of $\rI_{W(\fg_k,\gh_k)}(\fh_{k,\R})|_{\fh_{n,\R}}$ are polynomials in the
$x_j^2$, and the Pfaffian $($square root of the determinant$)$
is not in the image of the restriction map $\rI_{W(\fg_k,\fh_k)}(\fh_{k,\R})\to
\rI_{W(\fg_n,\fh_n)}(\fh_{n,\R})$.

\item With the assumptions from {\rm (2)} let $\sigma_k$ be as above and let
$\widetilde{W}_k(\fg_k,\fh_k)=
W (\fg_k,\fh_k)\ltimes \{1,\sigma_k\}$. Then $\widetilde{W} (\fg_k,\fh_k)$ is a finite reflection group, and
\[\widetilde{W} (\fg_n,\fh_n)=W_{\fh_n} (\fg_k,\fh_k)|_{\fh_{n,\R}}=
\widetilde{W}_{\fh_n}(\fg_k,\fh_k)|_{\fh_n}\, .\]

\item With the assumptions from {\rm (2)} we have
\[\rI_{W(\fg_k,\fh_k)}^{\text{even}}(\fh_{k,\R}) =
\rI_{\widetilde{W}(\fg_k,\fh_k)} (\fh_{k,\R})\]
and
\[ \rI_{W(\fg_k,\fh_k)}(\fh_{k,\R})|_{\fh_{n,\R}}=\rI_{\widetilde W(\fg_k,\fh_k)}(\fh_{k,\R})|_{\fh_{n,\R}}=\rI_{\widetilde W(\fg_n,\fh_n)}(\fh_{n,\R})\, .\]
\end{enumerate}
\end{theorem}

\begin{remark}\label{Re-RestrictionDoesNotWork}
{\rm If $\fg_k=\lsl(k+1,\C)$ and $\fg_n$ is constructed
from $\fg_k$ by removing any $n-k$ simple roots from the Dynkin diagram
of $\fg_k$, then Theorem \ref{th-AdmExt}(1) remains
valid because all the Weyl groups are permutation groups.
On the other hand, if $\fg_k$ is of type $B_k,C_k$, or $D_k$ ($k\geqq 3$)
and if $\fg_n$ is constructed from $\fg_k$ by removing at least one simple
root $\alpha_i$ with $k-i\geqq 2$,
then $\fg_n$ contains at least one
simple factor $\fl$ of type $A_\ell$, $\ell \geqq 2$. Let
$\fa$  be a Cartan subalgebra of $\fl$. Then the
restriction of the Weyl group of $\fg_k$ to
$\fa_\R$ will contain $-\mathrm{id}$.  But $-\mathrm{id}$ is not in the
Weyl group $W(\gs\gl(\ell+1,\C))$, and the restriction of the
invariants will only contain even polynomials.
Hence the conclusion Theorem \ref{th-AdmExt}(1) fails in this case.}
\hfill $\diamondsuit$
\end{remark}

We also note the following consequence of the definition of propagation.
It is implicit in the diagrams following that definition.
\begin{lemma}\label{le-RestrictionOfSimpleRoots} Assume that $\fg_k$
propagates $\fg_n$. Let $\fh_k$ be a Cartan subalgebra
of $\fg_k$ such that $\fh_n=\fh_k\cap \fg_n$ is a Cartan subalgebra of
$\fg_n$. Choose positive systems
$\Delta^+(\fg_k,\fh_k)\subset \Delta(\fg_k,\fh_k)$ and
$\Delta^+(\fg_n,\fh_n)\subset \Delta(\fg_n,\fh_n)$
such that $\Delta^+(\fg_n,\fh_n)\subseteqq
\Delta^+(\fg_k,\fh_k)|_{\fh_n}$.
Then we can number the simple roots such that
$\alpha_{n,j}=\alpha_{k,j}|_{\fh_n}$ for $j=1,\ldots ,\dim \fh_n$.
\end{lemma}

\section{Symmetric Spaces}\label{sec3}
\noindent
Now we discuss restriction of invariant polynomials related to
Riemannian symmetric spaces.
Let $M=G/K$ be a Riemannian symmetric space of compact or noncompact
type. Thus $G$ is a connected semisimple Lie group with an involution
$\theta$ such that
\[(G^\theta)_o\subseteqq K\subseteqq G^\theta\]
where $G^\theta =\{x\in G\mid \theta (x)=x\}$ and the subscript ${}_o$ denotes
the connected component containing the identity element. If
$G$ is simply connected then $G^\theta$ is connected and
$K=G^\theta$. If $G$ is noncompact and with finite
center, then $K\subset G$ is a \textit{maximal
compact} subgroup of $G$, $K$ is connected, and $G/K$ is simply
connected.

Denote the Lie algebra of $G$ by $\fg$. Then
$\theta$ defines an involution $\theta : \fg\to \fg$, and
$\fg=\fk\oplus \fs$
where $\fk=\{X\in\fg\mid \theta(X)=X\}$ is the Lie algebra of $K$ and
$\fs=\{X\in \fg\mid \theta (X)=-X\}$.

Cartan Duality is the bijection between simply connected
symmetric spaces of noncompact type and those of compact type defined by
$\fg = \fk\oplus \fs \leftrightarrow \fk\oplus i\fs = \fg^d$.
We denote it by $M\leftrightarrow M^d$.

Fix a maximal abelian subset $\fa\subset \fs$.  If  $\alpha \in\fa^*_\C$ we write
$\fg_{\C,\alpha} =\{X\in\fg_\C \mid [H,X]=\alpha (H)X \text{ for all }
H\in \fa_\C\}$, and if
$\fg_{\C,\alpha}\not=\{0\}$ then $\alpha$ is a (restricted) root. Denote
by $\Sigma (\fg,\fa)$ the set of roots.  If $M$ is of noncompact type, then
all the roots are in the real dual space $\fa^*$ and
$\fg_{\C,\alpha}=\fg_\alpha +i\fg_\alpha$, where
$\fg_\alpha =\fg_{\C,\alpha}\cap \fg$. If $M$ is of compact type, then the
roots take pure imaginary values on $\fa$,
$\Sigma (\fg,\fa)\subset i\fa^*$, and $\fg_{\C,\alpha}\cap \fg=\{0\}$. The
set of roots is preserved under duality
where we view those roots as $\C$--linear functionals on $\ga_\C$.

Let $\Sigma_{1/2}(\fg,\fa )=\{\alpha \in\Sigma(\fg,\fa ) \mid \tfrac{1}{2 }\alpha\not\in \Sigma (\fg, \fa)\}$.
Then $\Sigma_{1/2}(\fg, \fa )$ is a root system in the usual sense and the Weyl
group corresponding to $\Sigma (\fg, \fa)$ is the same as the Weyl group generated
by the reflections $s_\alpha$, $\alpha \in \Sigma_{1/2}(\fg, \fa)$.
Furthermore, $M$ is
irreducible if and only if $\Sigma_{1/2}(\fg, \fa )$ is irreducible, i.e., can not be
decomposed into two mutually orthogonal root systems.

Let $\Sigma^+(\fg, \fa)\subset \Sigma (\fg, \fa)$ be a positive system and
$\Sigma^+_{1/2}(\fg,\fa )=\Sigma^+ (\fg,\fa )\cap \Sigma_{1/2} (\fg,\fa)$. Then
$\Sigma^+_{1/2}(\fg,\fa )$ is a positive root system in
$\Sigma_{1/2} (\fg,\fa)$. Denote
by $\Psi_{1/2} (\fg,\fa)$ the set of simple roots in $\Sigma_{1/2}^+(\fg,\fa)$. Then
$\Psi_{1/2} (\fg,\fa )$ is a basis for $\Sigma (\fg,\fa )$.

The list of irreducible symmetric spaces is given by the following table.
The indices $j$ and $k$ are related by $k=2j+1$.  In the fifth
column we list the realization of $K$ as a subgroup of the compact
real form.  The second column indicates the type of the root system
$\Sigma_{1/2}(\fg,\fa)$.
(More detailed information is given by the Satake--Tits diagram for $M$;
see \cite{Ar1962} or \cite[pp. 530--534]{He1978}.
In that classification the case $\SU (p,1)$, $p\geqq 1$, is denoted by $AIV$,
but here it appears in $AIII$.
The case $\SO (p,q)$, $p+q$ odd, $p\ge q>1$, is denoted by $BI$ as in
this case the Lie algebra $\fg_\C=\so (p+q,\C)$ is of type $B$.
The case $\SO (p,q)$, with $p+q$ even, $p\ge q>1$ is denoted by $DI$ as
in this case $\fg_\C$ is of type $D$. Finally, the case $\SO (p,1)$, $p $
even, is denoted by $BII$ and $\SO (p,1)$, $p$ odd, is denoted by $DII$.)

{\footnotesize
\begin{equation}\label{symmetric-case-class}
\begin{tabular}{|c|c|l|l|l|c|c|} \hline
\multicolumn{7}{| c |}
{}\\
\multicolumn{7}{| c |}
{\normalsize{Irreducible Riemannian Symmetric $M = G/K$, $G$ classical,
$K$ connected}}\\
\multicolumn{7}{| c |}
{}\\
\hline \hline
\multicolumn{1}{|c}{} & & \multicolumn{1}{c}{$G$ noncompact}&
    \multicolumn{1}{|c}{$G$ compact} &
        \multicolumn{1}{|c}{$K$} &
        \multicolumn{1}{|c}{Rank$M$} &
        \multicolumn{1}{|c|}{Dim$M$} \\ \hline \hline
$1$ & $A_j$ &$\mathrm{SL}(j,\C)$ &$\SU (j)\times \SU(j)$ & $\diag\, \SU(j)$ & $j-1$ & $j^2-1$ \\ \hline
$2$ & $B_j$&$\SO (2j+1,\C)$&  $\SO (2j+1)\times \SO (2j+1)$ & $\diag\, \SO (2j+1)$ &
    $j$ & $2j^2+j$ \\ \hline
$3$ & $D_j$&$\SO (2j,\C)$ & $\SO (2j)\times \SO (2j)$ & $\diag\, \SO(2j)$ &
    $j$ & $2j^2-j$ \\ \hline
$4$ & $C_j$&$\Sp (j,\C)$&$\Sp (j)\times \Sp (j)$ & $\diag\,\Sp (j)$ & $j$ & $2j^2+j$ \\ \hline
$5$ & $AIII$& $\SU (p,q)$&$\SU(p+q)$ & $\mathrm{S}(\U (p)\times \U ( q))$ &
    $\min(p,q)$ & $2pq$ \\ \hline
$6$ &$AI$ &$\mathrm{SL}(j,\R)$& $\SU (j)$ & $\SO (j)$ & $j-1$ & $\tfrac{(j-1)(j+2)}{2}$ \\ \hline
$7$ &$AII$ &$\SU^*(2j) = SL(j,\H)$& $\SU (2j)$ & $\Sp (j)$ & $j-1$ & $2j^2-j-1$  \\ \hline
$8$ &$BDI$&$\SO_o (p,q)$ &$\SO (p+q)$ & $\SO (p) \times \SO (q)$ &
    $\min(p,q)$ & $pq$  \\ \hline
$9$ &$DIII$&$\SO^*(2j)$ &$\SO (2j)$ & $\U (j)$ & $[\tfrac{j}{2}]$ & $j(j-1)$ \\ \hline
$10$ &$CII$&$\Sp(p,q)$ &$\Sp (p+q)$ & $\Sp (p) \times \Sp (q)$ &
    $\min(p,q)$ & $4pq$  \\ \hline
$11$ & $CI$& $\Sp (j,\R)$  &$\Sp (j)$ & $\U (j)$ & $j$ & $j(j+1)$  \\ \hline
\end{tabular}
\end{equation}
}

Only in the following cases do we have
$\Sigma_{1/2}(\fg,\fa )\not= \Sigma (\fg,\fa)$:
\begin{itemize}
\item $AIII$ for $1 \leqq p < q$,
\item $CII$ for $1 \leqq p < q$, and
\item $DIII$ for $j$ odd.
\end{itemize}
In those three cases there is exactly one simple root with
$2\alpha \in\Sigma (\fg,\fa )$
and this simple root is  at the
right end of the Dynkin diagram for $\Psi_{1/2} (\fg,\fa )$. Also, either
$\Psi_{1/2} (\fg,\fa )=\{\alpha\}$ contains one simple root or
$\Psi_{1/2}(\fg,\fa )$ is of type $B_r$
where $r=\dim \fa$ is the rank of $M$.

Finally, the only two cases where $\Psi_{1/2} (\fg,\fa )$ is of type $D$ are
the case $\SO (2j,\C)/\SO(2j)$ or  the split case
$\SO_o(p,p)/\SO (p)\times \SO (p)$.
In particular, if $\Psi_{1/2}(\fg,\fa)$ is of type $D$ then $\fa$ is a Cartan subalgebra of $\fg$.

Let $G/K$ be an irreducible symmetric space of compact or non-compact type. As
before let $\fa\subset \fs$ be maximal abelian. Let $\fh$ be a Cartan subalgebra of
$\fg$ containing $\fa$. Then $\fh=(\fh\cap \fk) \oplus \fa$. Let
$\Delta (\fg,\fh)$, $\Sigma (\fg,\fa)$, and $\Sigma_{1/2}(\fg,\fa)$ denote the corresponding
root systems and $W (\fg,\fh)$ respectively $W (\fg,\fa)$ the Weyl group
corresponding to $\Delta (\fg,\fh)$ respectively $\Sigma (\fg,\fa)$. We define
an extension of those Weyl groups $\widetilde{W}(\fg,\fh)$ and $\widetilde{W}(\fg,\fa)$ as before.

\medskip
Note that $\widetilde{W}(\fg,\fa) = W (\fg,\fa)$  with only two exceptions:
(i) the case where $M$ locally isomorphic to
$\SO (2j,\C)/\SO (2j)$ (with $\fh = \fa_\C$) or its compact dual
$(SO(2j)\times SO(2j))/\diag\, SO(2j)$ (with $\fh \cong \fa \oplus \fa$),
and (ii) the case where $\SO_o(j,j)/\SO (j)\times \SO(j)$
or its compact dual $\SO (2j)/\SO (j)\times \SO (j)$ with $\fh = \fa$.

\begin{theorem}\label{th-IhIa} Let $G/K$ be a symmetric space of
compact or non-compact type (no Euclidean factors).
In the above notation,
$\widetilde{W}(\fg,\fa)= \widetilde{W}_\fa (\fg,\fh)|_{\fa}$ and
the restriction map
$I_{\widetilde{W}(\fg,\fh)}(\fh_\R )\to I_{\widetilde{W}(\fg,\fa)}(\fa )$ is
surjective.
\end{theorem}
\begin{proof} We can assume that $G/K$ is irreducible. If neither
$\Delta (\fg,\fh)$ nor $\Sigma (\fg,\fa)$ is of type $D$ this is
Theorem 5 from \cite{He1964}. According to the above discussion, the only
cases where $\Sigma (\fg,\fa)$ is of type $D$ are where $\Delta (\fg,\fh)$
is also of type $D$ and $\fa = \fh_\R$, or $\fa$ is the diagonal in
$\fh \cong \fa \oplus \fa$, or $\fa = \fh$.  The statement is clear when
$\fa$ is $\fh$ or $\fh_\R$.  If $\fa$ is the diagonal in $\fh \cong
\fa \oplus \fa$ then $\widetilde{W}_\fa (\fg,\fh)$ is the diagonal in
$\widetilde{W}(\fg,\gh) \cong
\widetilde{W}(\fg,\fa) \times \widetilde{W}(\fg,\fa)$, hence again is
$\widetilde{W}(\fg,\fa)$.

Now suppose that neither $\Delta (\fg,\fh)$ nor $\Sigma_{1/2} (\fg,\fa)$
is of type $D$.  Then $\widetilde{W}(\fg,\fa)=W(\fg,\fa)$ consists of all
permutations with sign changes (with respect to the correct basis). The
claim now follows from  the explicit calculations in
\cite[pp. 594, 596]{He1964}.
\end{proof}

Let $M_k=G_k/K_k$ and $M_n=G_n/K_n$ be irreducible
symmetric spaces of compact or noncompact type.  We say that
$M_k$ \textit{propagates} $M_n$, if $G_n\subseteqq G_k$, $K_n=K_k\cap G_n$,
and either $\fa_k=\fa_n$ or choosing $\fa_n\subseteqq \fa_k$ we
only add simple roots to the left end
of the Dynkin diagram for $\Psi_{1/2}(\fg_n,\fa_n)$ to obtain the Dynkin diagram
for $\Psi_{1/2}(\fg_k,\fa_k)$.
So, in particular $\Psi_{1/2}(\fg_n,\fa_n)$ and
$\Psi_{1/2}(\fg_k, \fa_k)$ are of the same type.
In general, if
$M_k$ and $M_n$ are Riemannian symmetric spaces
of compact or noncompact type, with universal covering
$\widetilde{M_k}$ respectively $\widetilde{M_n}$, then $M_k$
\textit{propagates}
$M_n$ if we can enumerate the irreducible factors of $\widetilde{M}_k= M_k^1\times
\ldots \times M_k^j$ and $\widetilde{M}_n=M_n^1\times \ldots \times M_n^i$, $i \leqq j$
so that $M_k^s$ propagates $M_n^s$ for $s=1,\ldots ,i$.
Thus, each $M_n$ is, up to covering, a product of irreducible factors
listed in Table \ref{symmetric-case-class}.

In general we can construct infinite sequences of propagations by moving
along each row in Table \ref{symmetric-case-class}. But there are also  inclusions like
$\mathrm{SL} (n,\R )/\SO (n)\subset \mathrm{SL} (k,\C)/\SU (k)$ which
satisfy the definition of propagation.

When $\fg_k$ propagates $\fg_n$, and $\theta_k$ and
$\theta_n$ are the corresponding involutions with
$\theta_k|_{\fg_n} = \theta_n$, the corresponding eigenspace decompositions
$\fg_k=\fk_k\oplus \fs_k$ and $\fg_n=\fk_n\oplus \fs_n$
give us
\[
\fk_n=\fk_k\cap \fg_n\, ,\quad \text{and}\quad \fs_n=\fg_n\cap \fs_k\, .\]
We recursively choose maximal commutative subspaces $\fa_k\subset \fs_k$ such
that $\fa_{n} \subseteqq \fa_k$ for $k\geqq n$.
Denote by $W(\fg_n,\fa_n)$ and $W(\fg_k,\fa_k)$ the corresponding Weyl
groups. The extensions $\widetilde{W}(\fg_k,\fa_k)$ and
$\widetilde{W}(\fg_n,\fa_n)$ are defined as just before Theorem \ref{th-IhIa}.
Let  $\rI(\fa_n)=\rI_{W (\fg_n,\fa_n)}(\fa_n)$,
$\rI_{\widetilde W (\fg_n,\fa_n)}(\fa_n)$, and
$\rI_{\widetilde W (\fg_k,\fa_k)}(\fa_k)$ denote the respective sets of Weyl
group invariant or $\widetilde{W}$--invariant polynomials on $\fa_n$ and
$\fa_k$. As before we let
\begin{equation}\label{eq-WknA}
W_{\fa_n}(\fg_k,\fa_k):=\{w\in W (\fg_k,\fa_k) \mid w(\fa_n)=\fa_n\}
\end{equation}
and define $\widetilde{W}_{\fa_n}(\fg_k,\fa_k)$ in the same way.

\begin{theorem}\label{th-AdmExtG/K} Assume that $M_k$ and $M_n$ are
symmetric spaces of compact or noncompact type and that $M_k$ propagates
$M_n$.

\noindent {\rm (1)} If $M_n$ does not contain any irreducible factor
with $\Psi_{1/2} (\fg_n,\fa_n)$ of type $D$, then
\begin{equation}\label{eq-RestrictionOfWeyl1}
W_{\fa_n}(\fg_k,\fa_k)|_{\fa_n} = W(\fg_n,\fa_n)
\end{equation}
and the restriction map $\rI(\fa_k)\to \rI(\fa_n)$ is surjective.

\noindent {\rm (2)} If $\Psi_{1/2}(\fg_n,\fa_n)$ is of type $D$ then
$$W (\fg_n,\fa_n)\subsetneqq
 W_{\fa_n} (\fg_k,\fa_k)|_{\fa_n} \text{ and }
\rI_{W(\fg_k,\fa_k)}(\fa_k)|_{\fa_n}\subsetneqq
  \rI_{W(\fg_n,\fa_n)}(\fa_{n}) .$$
On the other hand
$\widetilde{W}(\fg_n,\fa_n)= \widetilde{W}_{\fa_n} (\fg_k,\fa_k)|_{\fa_n}$
and $\rI_{\widetilde{W}(\fg_n,\fa_n)}(\fa_n) = \rI_{\widetilde{W}(\fg_k,\fa_k)}(\fa_k)|_{\ga_n}$.

\noindent {\rm (3)} In all cases
$\widetilde{W}(\fg_n,\fa_n)= \widetilde{W}_{\fa_n} (\fg_k,\fa_k)|_{\fa_n}$
and $\rI_{\widetilde{W} (\fg_k,\fa_k)}
(\fa_k)|_{\fa_n}= \rI_{\widetilde{W}(\fg_n,\fa_n)}(\fa_n)$.
\end{theorem}

\begin{proof} It suffices to prove this for each irreducible component of
$M_n$.  The argument of Theorem \ref{th-AdmExt} is valid here as well, and
our assertion follows.
\end{proof}

\section{Applications}\label{sec1}\setcounter{equation}{0}
\noindent
Our interest in restriction of Weyl groups and polynomial invariants came from the study of
projective limits of function of exponential growth. It turned out that the main step in showing that that the projective limit is non zero one  needed to understand the restriction of invariant polynomials and Weyl groups. We refer to \cite{OW0} for those applications. Some of those results are also mentioned in \cite{DO10} in this volume and will use the notation from that article. We assume that $M=G/K$ is a symmetric space of the noncompact type. We keep the notation from the previous sections. In particular, $\Sigma^+=
\Sigma^+(\fg,\fa)$ is a positive system of restricted roots. Let
\[\fn:=\bigoplus_{\alpha\in\Sigma^+}\fg_\alpha\text{ and } \fp:=\fm\oplus \fa\oplus \fn\, .\]
Then $\fn $ is a nilpotent Lie algebra and $\fp=\fn_\fg (\fn )$ is a minimal parabolic subalgebra. The corresponding minimal parabolic subgroup is $P=MAN$ with $M=Z_K (\fa )$, $A=\exp \fa$, and $N=\exp (\fn)$. We have the Iwasawa decomposition $G=KAN \simeq K\times A\times N$. Write $x=k(x)a(x)n(x)$ for the unique decomposition of $x$. This implies that
$B:=G/P=K/M$ and $G$ acts on $B$ by $x\cdot kM=k(xk)M$.

If $a=\exp (H)\in A$ and $\lambda\in \fa_\C^*$ then $(man)^\lambda = a^\lambda := e^{\lambda (H)}$. Let $\rho :=\frac{1}{2}\sum_{\alpha\in\Sigma^+} \dim \fg_\alpha \alpha\in \fa^*$. We normalize the invariant measures so that $K$ has measure one, $\int_{N} a(\theta( n))^{-2\rho}\, dn=1$, and the measure on $A$ and $\fa^*$ are normalized so that the Fourier inversion holds without constant. Finally
$\int_G f(g)\, dg=\int_K\int_A \int_N f (kan)a^{-2\rho} \, dndadk$, $f\in C_c(G)$. The spherical function with spectral parameter $\lambda\in\fa_\C^*$ is defined by
\begin{equation}\label{def-spherical}
\varphi_\lambda (x)=\int_G a(x^{-1}k)^{-\lambda - \rho}\, dk\, .
\end{equation}
We have $\varphi_\lambda =\varphi_\mu$ if and only if there exists $w\in W$ such that
$w\lambda =\mu$.

The spherical Fourier transform is defined by
\[\widehat{f}(\lambda )=\int_G f(x)\varphi_{-\lambda } (x)\, dx\, ,\quad f\in C_c^\infty (G/K)^K\, .\]
Then $\widehat{f}$ is a holomorphic Weyl group invariant function on $\fa_\C^*$.
Furthermore
\[f(x)=\int_{i\fa^*} \widehat{f}(\lambda )\varphi_\lambda (x)\, \frac{d\lambda }{\# W |c(\lambda )|^2}\]
and the Fourier transform extends to an unitary isomorphism
\[L^2(M)^K\simeq L^2\left(i\fa^*,\frac{d\lambda}{\# W|c (\lambda )|^2}\right)\, .\]
Here $c(\lambda)$ denotes the Harish-Chandra $c$-function.
We will also write $\cF(f)$ for $\widehat{f}$.

We start with the following lemma.

A connected semisimple Lie group $G$ is {\it algebraically simply connected}
if it is an analytic subgroup of the connected simply connected group
$G_\C$ with Lie algebra $\fg_\C$.  Then the analytic subgroup $K$ of $G$
for $\gk$ is compact, and every automorphism of $\fg$ integrates to an
automorphism of $G$.

\begin{lemma}\label{le-AutDynkDiagrG} Let $G/K$ be a Riemannian
symmetric space of noncompact type with $G$ simple and algebraically
simply connected.  Suppose that $\fa$ is a Cartan subalgebra of $\fg$,
i.e., that $\fg$ is a split real form of $\fg_\C$. If $\sigma :\fa\to \fa$ is
a linear isomorphism such that $\sigma'$ defines an automorphism of
the Dynkin diagram of $\Psi (\fg,\fa)$, then there exists a
automorphism $\widetilde{\sigma} :G\to G$ such that
\begin{enumerate}
\item $\widetilde{\sigma}|_\fa=\sigma$ where by abuse of notation we write
$\widetilde{\sigma}$ for $d\widetilde{\sigma}$,
\item $\widetilde{\sigma}$ commutes with the
the Cartan involution $\theta$, and in particular
$\widetilde{\sigma}(K)=K$,
\item $\widetilde{\sigma}(N)=N$.
\end{enumerate}
\end{lemma}

\begin{proof} The complexification of $\fa$ is a Cartan subalgebra
$\fh$ in $\fg_\C$ such that $\fh_\R=\fa$. Let
$\{Z_\alpha\}_{\alpha\in\Sigma (\fg,\fa)}$ be a
Weyl basis for $\fg_\C$ (see, for example, \cite[page 285]{V1974}).
Then (see, for example, \cite[Theorem 4.3.26]{V1974}),
\[\fg_0=\fa \oplus \bigoplus_{\alpha\in \Delta (\fg,\fh)}\R Z_\alpha
\]
is a real form of $\fg_\C$. Denote by $B$ the Killing form
of $\fg_\C$. Then $B(Z_\alpha ,Z_{-\alpha})=-1$ and
it follows that $B$ is positive definite on $\fa$ and on
$\bigoplus_{\alpha\in\Sigma^+(\fg,\fa)}\R (Z_\alpha- Z_{-\alpha})$,
and negative definite on
$\bigoplus_{\alpha\in\Sigma^+(\fg,\fa)}\R (Z_\alpha + Z_{-\alpha})$.
Hence, the map
$$\theta|_{\fa }=- \id \text{ and } \theta(Z_\alpha) = Z_{-\alpha}$$
defines a Cartan involution on $\fg_0$ such that the Cartan subalgebra
$\fa$ is contained in the corresponding $-1$ eigenspace $\fs$. As there
is (up to isomorphism) only one real form of $\fg_\C$ with Cartan
involution
such that $\fa\subset \fs$ we can assume
that $\fg=\fg_0$ and that the above Cartan involution $\theta$ is the
the one we started with.

Going back to the proof of \cite[Lemma 4.3.24]{V1974}
the map defined by
\[\widetilde{\sigma}|_\fa = \sigma \, \quad \text{and}
\quad \widetilde{\sigma}(Z_\alpha ) =Z_{\sigma {\alpha}}\]
is a Lie algebra isomorphism $\widetilde{\sigma}:\fg\to \fg$.
But then
$$
\widetilde{\sigma}(\theta (Z_\alpha)) = \widetilde{\sigma}(Z_{-\alpha})
= Z_{\sigma (-\alpha )} = Z_{-\sigma (\alpha )} =
\theta (\widetilde{\sigma }(Z_\alpha)).
$$
Finally, $\theta|_\fa = -\id$ and it follows that $\widetilde{\sigma }$
and $\theta$ commute. As
\[\fk =\bigoplus_{\alpha\in\Sigma^+(\fg,\fa)}\R (Z_\alpha +\theta
(Z_{\alpha}))\]
and $\sigma (\Sigma^+(\fg,\fa))=\Sigma^+(\fg,\fa)$ it follows that
$\widetilde{\sigma} (\fk )=\fk$.

As $\sigma(\Sigma^+(\fg,\fa))=\Sigma^+(\fg,\fa)$ it follows that
$\widetilde{\sigma}(\fn)=\fn$.

As $G$ is assumed to be algebraically simply connected, there is an
automorphism of $G$ with differential $\widetilde{\sigma}$. Denote this
automorphism also by $\widetilde{\sigma}$. It is clear that
$\widetilde{\sigma}$ satisfies the assertions of the lemma.
\end{proof}

Define an involution $\widetilde\sigma$ on $G$ in the following way: If $G_j/K_j$ is an irreducible factor of $M=G/K$ then $\widetilde \sigma|_{G_j}$ is the identity if $G_j/K_j$ is not of type $D$, otherwise it is the involution from Lemma \ref{le-AutDynkDiagrG}. Then we define $\widetilde{G}=G\rtimes \{1,\widetilde{\sigma}\}$ and $\widetilde{K}=K\rtimes \{1,\widetilde{\sigma}\}$. Note that $M=G/K =\widetilde{G}/\widetilde{K}$.

\begin{theorem}\label{th-SphericalFctTildeWInv} Let $\lambda\in\fa_\C^*$ and $x\in M$. Then
\[\varphi_\lambda ( \widetilde{\sigma}(x))=\varphi_{\sigma(\lambda
)}(x)\, .\]
If $f\in L^2(M)^{\widetilde{K}}$ then
$\widehat {f}$ is $\sigma$-invariant.
\end{theorem}

\begin{proof} Write $x=kan$, then $\wts (x)=\wts (k)\wts (a) \wts (n)$. Thus
$a(\wts (x))=\wts (a(x))$.  By
(\ref{def-spherical}) and the fact that $\sigma(\rho)=\rho$ and that the invariant measure on $K$ is $\wts$-invariant
we get
\begin{eqnarray*}
\varphi_\lambda (\wts (x))&=&
\int_K a(\widetilde{\sigma}(x^{-1})k)^{-\lambda -\rho}\, dk\\
&=&\int_K (\widetilde{\sigma}( a(x^{-1}k)))^{-\lambda -\rho} \, dk\\
&=&\int_K a(x^{-1}k)^{-\sigma \lambda -\rho}\, dk\\
&=&\varphi_{\sigma(\lambda )}(x)\, .
\end{eqnarray*}
The remaining statements are now clear.
\end{proof}

Fix a positive definite $K$--invariant bilinear form
$\langle \cdot ,\cdot \rangle $ on $\fs$. It defines an
invariant Riemannian structure on $M$ and hence also an invariant
metric $d(x,y)$. Let $x_o=eK\in M$ and for $r>0$ denote by $B_r=B_r(x_o)$
the closed ball
\[B_r =\{x\in M\mid d(x,x_o)\leqq r\}\, .\]
Note that $B_r$ is $\widetilde{K}$--invariant. Denote by $C_r^\infty (M)^{\widetilde{K}}$ the space of
smooth $\widetilde{K}$--invariant functions on $M$ with support in $B_r$.
The restriction map $f\mapsto f|_A$ is a bijection from
$C_r^\infty (M)^{\widetilde{K}}$ onto $C_r^\infty (A)^{\widetilde W }$ (using the
obvious notation).

For a finite dimensional Euclidean vector space $E$ and a closed subgroup $W$ of $\mathrm{O} (E)$ let $\PW_r(E_\C)^{W}$ be the space of holomorphic functions $F: E_\C \to \C$ such that for all $k\in \N$
\[\sup_{z\in E_\C}(1+|z|)^ke^{-r|\Im z|}|F(z)|<\infty\]
and $F(w\cdot z)=F(z)$ for all $z\in E_\C$ and $w\in W$. In particular
$\PW_r (\fa^*_\C)^{\widetilde{W}}$ is well defined.
The following is a simple modification
of the Paley-Wiener theorem of Helgason \cite{He1966,He1984} and Gangolli
\cite{Ga1971}; see \cite{OP2004} for a short overview.

\begin{theorem}[The Paley-Wiener Theorem]\label{th-IsomorphismNonCompact1} 
The Fourier transform defines  bijections
$$
 C^\infty_r (M)^{\widetilde{K}} \cong \PW_r(\fa_\C^*)^{\widetilde{W} }\, .
$$
\end{theorem}

We assume now that $M_k$ propagates $M_n$, $k\geqq n$. The index $j$ refers
to the symmetric space $M_j$, for a function $F$ on $\fa_{k,\C}^*$ let $P_{k,n}(F):=
F|_{\fa_{n,\C}}$. We fix a compatible $K$--invariant
inner products on $\fs_n$ and $\fs_k$, i.e., for all $X,Y\in\fs_n\subseteqq \fs_k$ we have
\[\langle X,Y\rangle_k=\langle X,Y\rangle_n\, .\]
We refer to \cite{OW0} for the application to injective sequences of symmetric spaces, for the injective limit of symmetric spaces of the noncompact type, see also the overview \cite{DO10} in this volume.

\begin{theorem}[\cite{OW0}]\label{th-IsomorphismNonCompact2}
Assume that $M_k$ propagates $M_n$. Let $r>0$. Then the following holds:
\begin{enumerate}
\item The map $P_{k,n}:
\PW_r (\fa_{k,\C}^*)^{\widetilde{W}(\fg_k,\fa_k)} \to
\PW_r(\fa_{n,\C}^*)^{\widetilde{W}(\fg_n,\fa_n)}$ is surjective.
\item The map $C_{k,n}=\cF^{-1}_n\circ P_{k,n}\circ \cF_k:C^\infty_{r}
(M_k)^{\widetilde K_k}
\to C^\infty_{r}(M_n)^{\widetilde K_n}$ is surjective.
\end{enumerate}
\end{theorem}

Let us explain the connection with Theorem \ref{th-AdmExtG/K}.
For that let $F\in \PW_r(\fa_{n,\C}^*)^{\widetilde{W}_n}$, where $\wtW_n=\wtW (\fg_n,\fa_n)$.
Then, according to a result of Cowling \cite{cowling} there exists a $G\in \PW_r(\fa_{k,\C}^*)^{\wtW_k}$ such that $G|_{\fa_{n,\C}^*}=F$. We can assume that $G$ is
invariant under $\wtW_{k,n}=\{w\in \wtW_k\mid w(\fa_n)=\fa_n\}$. As $\wtW_k$ is a finite reflection group  it follows by \cite{Rais} that there exists $G_1,\ldots ,G_{r_k}\in
\PW_r(\fa_{k,\C}^*)^{\wtW_k}$ and $p_1,\ldots ,p_{r_k}\in \rI_{\wtW_{k,n}}(\fa_{k})$ such that
\[G=p_1G_1+\ldots +p_{r_n}G_{r_n}\, .\]
As $p_j|_{\fa_n}\in \rI_{\wtW_n}(\fa_{n})$ Theorem \ref{th-AdmExtG/K} there exists $q_j\in
\rI_{\wtW_k}(\fa_k)$ such that $q_j|_{\fa_n}=p_j|_{\fa_n}$. But then
$H:=q_1G_1+\ldots + q_{r_k}G_{r_k}\in\PW_{r}(\fa_{k,\C}^*)^{\wtW_k}$ and $H|_{\fa_{n,\C}^*}=
F$ showing that the restriction map is surjective.

It is well known, \cite[Thm 5.13,p.300]{He1984}, that if $M=G/K$ is a Riemannian symmetric space of the noncompact type then there exists an algebra isomorphism $\Gamma :\mathbb{D}(M)\to \rI_W (\fa )$, where $\mathbb{D}(M)$ is the algebra of invariant differential operators, such that
\[D\varphi_\lambda =\Gamma (\lambda) \varphi_\lambda\quad \text{for all } \lambda\in \fa_{n,\C}^*\, .\]

Restricting $\Gamma $ to $\widetilde{\mathbb{D}}(M)$, the algebra of $\widetilde{G}$-invariant differential operators on $M$ then gives:
\begin{lemma} There exists an algebra isomorphism $\widetilde{\Gamma} : \widetilde{\mathbb{D}}(M)\to \rI_{\wtW}(\fa_n)$ such that for all $\lambda\in \fa_\C^*$  and $D\in\widetilde{\mathbb{D}}(M)$ we have
\[D\varphi_\lambda = \widetilde{\Gamma}(D)\varphi_\lambda\, .\]
\end{lemma}

\begin{theorem} Assume that  $M_k$ propagates $M_n$.  There exists a surjective algebra homomorphism $\Gamma_{k,n} : \widetilde{\mathbb{D}}(M_k)\to \widetilde{\mathbb{D}}(M_n)$ such that for all $f\in C^\infty_c(M_k)$ we have
\[C_{k,n} (Df)=\Gamma_{k,n}(D)C_{k,n}(f)\, .\]
\end{theorem}
\begin{proof} For $D\in \widetilde{\mathbb{D}}(M_k)$ define $\Gamma_{k,n} (D):=\Gamma_n^{-1}(\Gamma_k(D)|_{\fa_n})$. Then $\Gamma_{k,n}(D)\in \widetilde{\mathbb{D}}(M_n)$ and by Theorem \ref{th-AdmExtG/K} $\Gamma_{k,n} : \widetilde{\mathbb{D}}(M_k)\to \widetilde{\mathbb{D}}(M_n)$ is a surjective homomorphism.

Let $f\in C_c^\infty (M_k)$.
Then
\begin{eqnarray*}
C_{k,n} (Df)&=& \cF_{n}^{-1}(P_{k,n}(\cF_k(Df)))\\
&=& \cF_n^{-1}(\Gamma_k(D)|_{\fa_n}P_{k,n}(f)\\
&=& \Gamma_n^{-1}(\Gamma_k(D)|_{\fa_n})C_{k,n}(f)\\
&=&\Gamma_{k,n}(D)C_{k,n}(f)
\end{eqnarray*}
proving the theorem.
\end{proof}

Note, if we take $D$ to be the Laplacian $\Delta_k$ on $M_k$ then $\Gamma_{k} (D)=\lambda^2-|\rho_k|^2$ where $\lambda^2=\lambda_1^2+\ldots +\lambda_{r_k}^2$ where we write $\lambda =\lambda_1e_1^*+ \ldots +\lambda_{r_k}e_{r_k}^*$ with respect to an orhonormal basis of $\fa_k^*$. Thus
\[\Gamma_{k,n}(\Delta_k)=\Delta_n-(|\rho_k|^2-|\rho_n|^2)\, .\]

Let $M_k =\SO_o(1,k)/\SO (k)$ and $M_n =\SO (1,n)/\SO (n)$ then $\fa_k=\fa_n$, $\Sigma =\{\alpha,-\alpha\}$, $\rho_k =\frac{k}{2}\alpha$, and $\rho_n=\frac{n}{2}\alpha$. Normalizing the inner product so that $|\alpha |=1$ we get
\[|\rho_k|^2-|\rho_n|^2=\frac{1}{4}(k^2-n^2)\to \infty \text{ as } n,k\to \infty\, .\]
Hence in the limit $\Delta_\infty$ does not exists. However, the shifted Laplacian $\Delta_k-|\rho_k|^2$ has a limit as $k\to\infty$. It should be noted, that it is exactly this shifted Laplacian that plays a role in the wave equation on symmetric spaces of the noncompact type, see \cite{BOS95,OS92} and the reference therein. It is also interesting to note that in \cite{S07} the same $\rho$-shift was used in the spherical functions to study the heat equation on inductive limits of a class of symmetric spaces of the noncompact type.

\end{document}